\documentclass[10pt]{article} 
\usepackage[preprint]{tmlr}
\usepackage{siunitx}

\usepackage{amsmath,amsfonts,bm}

\def\eqref#1{equation~\ref{#1}}

\def\1{\bm{1}}

\def\rd{{\textnormal{d}}}

\DeclareMathAlphabet{\mathsfit}{\encodingdefault}{\sfdefault}{m}{sl}
\SetMathAlphabet{\mathsfit}{bold}{\encodingdefault}{\sfdefault}{bx}{n}

\newcommand{\E}{\mathbb{E}}

\newcommand{\R}{\mathbb{R}}

\usepackage{mathtools,amssymb,amsthm,amsmath}
\usepackage{graphicx,color}
\usepackage{subcaption}
\usepackage{booktabs}
\usepackage{float}
\usepackage[boxed]{algorithm2e}
\usepackage[hyphens]{url}
\usepackage{bm}

\usepackage[toc,page]{appendix}
\usepackage{hyperref}
\usepackage[capitalize]{cleveref}
\usepackage{bm}
\usepackage{xcolor}
\usepackage{enumitem}

\usepackage{bbm}
\usepackage{enumitem}
\usepackage{mathrsfs}
\usepackage{ragged2e}
\usepackage{tikz}

\newcommand{\cF}{\mathcal{F}}

\newcommand{\cP}{\mathcal{P}}

\newcommand{\cR}{\mathcal{R}}

\newcommand{\kl}[2]{\text{KL}(#1 \| #2)}

\newcommand*{\chis}[2]{\chi^2(#1\| #2)}

\newcommand*{\triplenorm}[1]{{\left\vert\kern-0.25ex\left\vert\kern-0.25ex\left\vert #1
    \right\vert\kern-0.25ex\right\vert\kern-0.25ex\right\vert}}

\newcommand{\Rd}{\mathbb{R}^d}
\renewcommand{\phi}{\varphi}
\newcommand{\sse}{\subseteq}

\newcommand*{\defeq}{\coloneqq}

\newcommand*{\dd}{\, \rd}

\newtheorem{lemma}{Lemma}

\newtheorem{theorem}{Theorem}

\newtheorem{proposition}{Proposition}

\newtheorem{remark}{Remark}

\usepackage{subcaption}

\usepackage{hyperref}
\usepackage{url}
\usepackage[normalem]{ulem}
\usepackage{dsfont}

\title{An Explicit Expansion of the Kullback-Leibler Divergence along its Fisher-Rao Gradient Flow}

\author{\name Carles Domingo-Enrich \email cd2754@nyu.edu \\
      \addr Courant Institute of Mathematical Sciences\\
      New York University
      \AND
      \name Aram-Alexandre Pooladian \email aram-alexandre.pooladian@nyu.edu \\
      \addr Center for Data Science\\ New York University}

\def\month{MM}  
\def\year{YYYY} 
\def\openreview{\url{https://openreview.net/forum?id=XXXX}} 

\begin{document}

\maketitle

\begin{abstract}
Let $V_* : \mathbb{R}^d \to \R$ be some (possibly non-convex) potential function, and consider the probability measure $\pi \propto e^{-V_*}$. 
When $\pi$ exhibits multiple modes, it is known that sampling techniques based on Wasserstein gradient flows of the Kullback-Leibler (KL) divergence (e.g. Langevin Monte Carlo) suffer poorly in the rate of convergence, where the dynamics are unable to easily traverse between modes. 
In stark contrast, the work of \cite{lu2019accelerating,lu2022birth}
has shown that the gradient flow of the KL with respect to the Fisher-Rao (FR) geometry exhibits a convergence rate to $\pi$ is that \textit{independent} of the potential function. 
In this short note, we complement these existing results in the literature by providing an explicit expansion of $\text{KL}(\rho_t^{\text{FR}}\|{\pi})$ in terms of $e^{-t}$, where $(\rho_t^{\text{FR}})_{t\geq 0}$ is the FR gradient flow of the KL divergence.
In turn, we are able to provide a clean asymptotic convergence rate, where the burn-in time is guaranteed to be finite.
Our proof is based on observing a similarity between FR gradient flows and simulated annealing with linear scaling, and facts about cumulant generating functions. 
We conclude with simple synthetic experiments that demonstrate our theoretical findings are indeed tight.
Based on our numerics, we conjecture that the asymptotic rates of convergence for Wasserstein-Fisher-Rao gradient flows are possibly related to this expansion in some cases.
\end{abstract}

\section{Introduction} \label{sec: intro}
Sampling from a distribution with an unknown normalization constant is a widespread task in several scientific domains. Namely, the goal is to generate samples from a probability measure
\begin{align*}
    \pi(x) \propto e^{-V_*(x)}\,,
\end{align*}
where $V_* : \Rd \to \R$ is some (possibly non-convex) potential function that is available for queries. In most cases, the target measure $\pi$ is only known up to the normalization constant. Applications of sampling from $\pi$ include Bayesian statistics, high-dimensional integration, differential privacy, statistical physics and uncertainty quantification; see  \cite{gelman1995bayesian,
robert1999monte,
mackay2003information,
johannes2010mcmc,
von2011bayesian,
kobyzev2020normalizing,sinhobook} for thorough treatments.

Recent interest in the task of sampling stems from the following paradigm: sampling is nothing but optimization over the space of probability measures \citep{wibisono2018sampling}. This interpretation is due to the connection between the celebrated work of Jordan, Kinderleher, and Otto \citep{jordan1998variational} and the Langevin diffusion dynamics given by
\begin{align}\label{eq: LD}
    \dd X_t = -\nabla V_*(X_t)\dd t + \sqrt{2} \dd B_t\,,
\end{align}
where $\dd B_t$ is Brownian motion.\footnote{This equation is to be understood from the perspective of It{\^o} calculus.} Indeed, the work of \cite{jordan1998variational} demonstrates that the path in the space of proabability measures given by the law of Eq. (\ref{eq: LD}) is the same as the Wasserstein gradient flow (i.e. steepest descent curve in the Wasserstein metric) of the Kullback-Leibler (KL) divergence 
\begin{align*}
\kl{\rho}{\pi} = \int \log \frac{\rho}{\pi} \dd \rho\,.
\end{align*}
 We write $(\rho_t^{\text{W}})_{t \geq 0} \sse \cP(\Rd)$ for the law of the path given by Eq. (\ref{eq: LD}) (see \cref{subsec:wasserstein} for a precise definition). 

A central problem in this area has been to bound the convergence rate of $\rho_t^{\text{W}}$ to $\pi$ in certain similarity metrics (e.g. the KL divergence itself, or the Wasserstein distance) under different conditions on $\pi$. These bounds translate to convergence rates for the Langevin Monte Carlo (LMC) sampling algorithm \citep{dalalyan2012sparse,vempala2019rapid,durmus2021analysis,chewi2022analysis}, upon accounting for discretization errors.

The classical result is as follows: assuming that $\pi$ satisfies a Log-Sobolev inequality (LSI) with constant $C_{\texttt{LSI}} > 0$, we obtain the following convergence rate \citep{stam1959some,gross1975logarithmic,Markowich99onthe} 
\begin{align}\label{eq: w_conv}
    \kl{\rho_t^{\text{W}}}{\pi} \leq \kl{\rho_0^{\text{W}}}{\pi}e^{-\frac{2t}{C_{\texttt{LSI}}}}\,,
\end{align}
which holds for all $t \geq 0$. Recall that $\pi$ satisfies a LSI if for all smooth test functions $g$,
\begin{align}\label{eq: LSI}
    \text{ent}_\pi(f^2) \leq 2 C_{\texttt{LSI}} \E_\pi\|\nabla f\|^2\,,
\end{align}
where $\text{ent}_\pi(g) \defeq \E_\pi(g\log g) - \E_\pi g \log \E_\pi g.$ For example, when $V_*$ $\alpha$-strongly convex, an LSI with $C_{\texttt{LSI}}=1/\alpha$ holds. 
LSI hold more generally, but sometimes with very large constants $C_{\texttt{LSI}}$. Indeed, for multimodal distributions such as mixtures of Gaussians, $C_{\texttt{LSI}}$ scales exponentially in the height of the potential barrier between modes \citep{Holley1987,arnold2000onconvex}.
This impacts convergence at the discrete-time level, and thus hinders our ability to generate samples using LMC.

Another geometry that gives rise to gradient flows over probability measures is the \textit{Fisher-Rao} (FR) geometry; see \cref{sec: FRGF} for definitions. Similar to the case of Wasserstein gradient flows, we let $(\rho_t^{\text{FR}})_{t \geq 0}$ be the FR gradient flow of the KL divergence. 
Recent work by Lu and collaborators has shown that the convergence $\rho_t^{\text{FR}} \to \pi$ occurs at a rate that is \textit{independent} of the potential function $V_*$. This is in stark contrast to the case of Wasserstein gradient flows, where the rate of convergence is intimately related to the structure of $V_*$ through the LSI constant. In their first work,  \cite{lu2019accelerating} show that for any $\delta \in (0,\tfrac14]$ there exists a $t_* \gtrsim \log(\delta^3)$ such that for all $t \geq t_*$,
\begin{align}\label{eq: fr_conv}
    \kl{\rho_t^{\text{FR}}}{\pi} \leq \kl{\rho_0^{\text{FR}}}{\pi}e^{-(2-3\delta)(t-t_*)}\,,
\end{align}
where they require a warm-start condition $\kl{\rho_0^{\text{FR}}}{\pi} \leq 1$, and assumption \textbf{(B)} (see \cref{sec: main}).
In \cite{lu2022birth}, the authors show that the KL divergence is always contracting under $(\rho_t^{\text{FR}})_{t\geq0}$ even in the absence of a warm-start, though with a worse rate. 
Combined, these two results provide the first continuous-time convergence rates of the gradient flow of the KL divergence under the FR geometry to $\pi$.

Merging both these geometries gives rise to the well-defined \textit{Wasserstein-Fisher-Rao} (WFR) geometry. The WFR geometry has recently been used to analyse the convergence dynamics of parameters of neural networks \cite{chizat2022sparse}, mean-field games \cite{rotskoff2019global}, and has shown to be useful in statistical tasks such as Gaussian variational inference \cite{lambert2022variational}, and identifying parameters of a Gaussian mixture model \cite{yan2023learning}. In the context of sampling, particle-based methods that follow dynamics governed by WFR gradient flow of the KL, written $(\rho_t^{\text{WFR}})_{t \geq 0}$, are known to escape the clutches of slow-convergence that plague the Wasserstein geometry. A simple observation \cite[Remark 2.4]{lu2022birth} gives the following continuous-time convergence rate for $t \geq t_*$:
\begin{align}\label{eq: wfr_conv}
    \kl{\rho_t^{\text{WFR}}}{\pi} \leq \min\{ \kl{\rho_t^{\text{W}}}{\pi}, \kl{\rho_t^{\text{FR}}}{\pi} \} \leq \kl{\rho_0^{\text{WFR}}}{\pi}\min\left\{ e^{-C_{\texttt{LSI}}t},e^{-(2-3\delta)(t-t_*)}\right\}\,,
\end{align}
where $\delta$ and $t_*$ are as in the FR convergence rate (\ref{eq: fr_conv}). Loosely speaking, this ``decoupled rate'' is a consequence of the Wasserstein and FR geometries being orthogonal to one another; this is made precise in \cite{gallouet2017jko}.

As elegant as this last connection may seem, the convergence rate in Eq. (\ref{eq: fr_conv}), and consequently Eq. (\ref{eq: wfr_conv}), should appear somewhat unsatisfactory to the reader. It raises the natural question of whether or not the factor of $\delta$ appearing in the rate is avoidable, and whether the upper bound in Eq. (\ref{eq: fr_conv}) is tight.

\subsection{Main contributions}
We close this gap for the KL divergence and any $q$-R{\'e}nyi divergence. Using a different proof technique than existing work, we prove the following asymptotic rate of convergence for the flow $(\rho_t^{\text{FR}})_{t \geq 0}$, namely for $t$ sufficiently large,
\begin{align}
\kl{\rho_t^{\text{FR}}}{\pi} = \tfrac12 \text{Var}_\pi\left( \log \frac{\rho_0^{\text{FR}}}{\pi} \right) e^{-2t} + O(e^{-3t})\,,
\end{align}
and a similar result holds for all $q$-R{\'e}nyi divergences. Our assumptions are weaker to that of prior work, and given that this is a tight asymptotic convergence rate, we conjecture that the assumptions are likely unavoidable in the large $t$ regime. 
Our proof technique provides an explicit expansion of $\kl{\rho_t^{\text{FR}}}{\pi}$ (and $q$-R{\'e}nyi) in terms of $e^{-t}$.
We supplement our finding with simulations for all three geometries, indicating that our convergence rate is in fact tight for Fisher-Rao gradient flows, and sheds light on possible conjectures for the convergence rate of WFR gradient flows.

\subsubsection*{Notation}
For a probability measure $\rho \in \cP(\Rd)$ and a function $f : \Rd \to \R$, we sometimes use the shorthand $\langle f \rangle_\rho \defeq \int f \dd \rho$. We let $\log(\cdot)$ denote the natural logarithm, and we use the standard shorthand notation $f = O(g)$, meaning there exists a constant $C > 0$ such that $f \leq Cg$.

\section{Background}\label{sec: background}
\subsection{Definitions}\label{sec: defs}
The study of gradient flows has a rich history in both pure and applied mathematics. 
The development of the relevant calculus to understand gradient flows is not the purpose of this note, and we instead provide a barebones introduction. 
However, we strongly recommend the interested reader consult standard textbooks on the topic, namely \cite{ambrosio2005gradient}, and the first chapter of \cite{sinhobook}.

Let $\cP(\Rd)$ be the space of probability measures over $\Rd$. A functional $\cF : \cP(\Rd) \to \R$ is defined on the space of probability measures, with $\rho \mapsto \cF(\rho) \in \R$. We call ${\delta \cF}(\rho)$ the \textit{first variation of $\cF$ at $\rho$} if for a signed measure $\eta$ such that $\int \dd \eta = 0$, it holds that 
\begin{align}
    \lim_{\varepsilon \to 0} \frac{\cF(\rho + \varepsilon \eta) - \cF(\rho)}{\varepsilon} = \int {\delta \cF}(\rho) \dd \eta\,.
\end{align}

The Kullback-Leibler (KL) divergence of a measure $\rho$ with respect to some fixed target measure $\pi$ is defined as $\kl{\rho}{\pi} = \int \log \frac{\rho}{\pi}\dd \rho$ for $\rho$ absolutely continuous with respect to $\pi$. For $\pi \propto e^{-V_*}$, the first variation of the KL divergence is given by
\begin{align}\label{eq: first_var_kl}
   {\delta \kl{\cdot}{\pi}}(\rho)(x) = \log \frac{\rho(x)}{\pi(x)} = \log \rho(x) + V_*(x) + \log Z_1\,,
\end{align}
where $Z_1$ is the normalizing constant for $\pi$.

A more general notion of dissimilarity between probability measures is the $q$-R{\'e}nyi divergence: for $q \in [1,\infty]$, we define $\cR_q(\rho\|\pi)$ to be the $q$-R{\'e}nyi divergence with respect to $\pi$, given by
\begin{align}
\cR_q(\rho\|\pi) \defeq \frac{1}{q - 1} \log \int \left(\frac{\rho}{\pi}\right)^{q} \dd \pi\,,
\end{align} 
for measures $\rho$ that are absolutely continuous with respect to $\pi$.
$\cR_q$ recovers the KL divergence in the limit $q\to 1$, and when $q=2$, $\cR_2(\rho\|\pi) = \log(\chis{\rho}{\pi}+1)$, where $\chi^2$ is the chi-squared divergence, written explicitly as 
\begin{align*}
\chis{\rho}{\pi} = \text{Var}_\pi\left(\frac{\rho}{\pi}\right) = \int \left(\frac{\rho}{\pi}\right)^2 \dd \pi - 1\,.
\end{align*}

\subsection{Gradient flows of the Kullback-Leibler divergence}\label{sec: GF_KL}
\subsubsection{Wasserstein gradient flow} \label{subsec:wasserstein}
In its \textit{dynamic formulation}, the 2-Wasserstein distance between two probability measures $\rho_0,\rho_1$ with bounded second moments can be written as \citep{villani2008optimal,benamou2000computational}
\begin{align}
    \mathrm{W}_2^2(\rho_0,\rho_1) \defeq \inf_{(\rho_t,v_t)} \int_0^1 \int \|v_t(x)\|^2 \rho_t(x) \dd x \dd t \quad \text{s.t.} \quad \partial_t  \rho_t + \nabla \cdot (\rho_t v_t) = 0\,,
\end{align}
where $(\rho_t)_{t \in [0,1]}$ is {a curve} of probability densities over $\Rd$, and $(v_t)_{t \in [0,1]}$ is {a curve} of $L^2(\Rd)^d$ vector fields. The constraint is known as the continuity equation, with endpoints $\rho_0$ and $\rho_1$.
For a functional $\cF: \cP(\Rd) \to \R$, the \textit{Wasserstein gradient flow} is the curve of measures $(\rho_t^\text{W})_{t \geq 0}$ that satisfies the continuity equation with the vector field replaced by the steepest descent under the Wasserstein geometry, 
\begin{align*}
    v_t = -\nabla_{W_2} \cF(\rho_t^{\text{W}}) \defeq \nabla {\delta \cF}(\rho_t^{\text{W}})\,,
\end{align*}
where the last equation is simply the (standard) spatial gradient of the first variation of $\cF$. Plugging in the expression for the first variation of the KL divergence (\ref{eq: first_var_kl}), we see that the law of the Langevin diffusion is given by $\rho_t^{\text{W}}$ which satisfies
\begin{align} \label{eq:W_flow}
    \partial_t \rho_t^{\text{W}} = 
    \nabla \cdot \left(\rho_t^{\text{W}} (\nabla \log \rho_t^{\text{W}} + \nabla V_*) \right)\,.
\end{align}
This equation may be rewritten as $\partial_t \rho_t^{\text{W}} = \nabla \cdot (\nabla V_* \rho_t) + \Delta \rho_t$, which one readily identifies as the Fokker-Planck equation for the potential $V_*$.
The equation describes the evolution of the distribution of a particle that moves according to the stochastic differential equation \ref{eq: LD}.
At the particle level, the key aspect of Wasserstein gradient flows is that they model particle \textit{transport}, and that makes them useful for high-dimensional applications such as LMC.
In what follows, we will sometimes abbreviate Wasserstein gradient flow to W-GF.

\subsubsection{Fisher-Rao gradient flow}\label{sec: FRGF}
The Fisher-Rao distance, or Hellinger-Kakutani distance, between probability measures has a long history in statistics and information theory \citep{hellinger1909neue,kakutani1948onequivalence}. It can be defined as \citep{bogachev2007measure,gallouet2017jko}
\begin{align*}
    \mathrm{FR}^2(\rho_0,\rho_1) \defeq \inf_{(\rho_t,r_t)} \int_0^1 \int r_t(x)^2 \rho_t(x) \dd x \dd t \quad \text{s.t.} \quad \partial_t  \rho_t = r_t \rho_t\,,
\end{align*}
where $(\rho_t)_{t\in[0,1]}$ is again a curve of probability measures, and $(r_t)_{t\in[0,1]}$ is {a curve} of $L^2(\Rd)$ functions. Together, they satisfy the prescribed equation, with endpoints equal to $\rho_0$ and $\rho_1$.
The Fisher-Rao gradient flow of the KL divergence, also known as \textit{Birth-Death dynamics}, is the curve of measures $(\rho_t^{\text{FR}})_{t \geq 0}$ that satisfies \citep{gallouet2017jko,lu2019accelerating}
\begin{align*}
    \partial_t \rho^{\text{FR}}_t = -\rho^{\text{FR}}_t \alpha_t\,, \quad \alpha_t \defeq  \log\frac{\rho^{\text{FR}}_t}{\pi} - \kl{\rho^{\text{FR}}_t}{\pi}\,.
\end{align*}
The first term adjusts mass (i.e. gives birth to or kills mass) according to the log-ratio of $\rho_t^{\text{FR}}$ and the target measure $\pi$. The last term preserves the total mass, so that $\rho_t^{\text{FR}} \in \cP(\Rd)$ for all time.

Expanding this equation, we have
\begin{align} \label{eq:FR_KL}
    \partial_t \rho_t^{\text{FR}}(x) = -\big( \log(\rho_t^{\text{FR}}(x)) + V_*(x) - \big\langle \log(\rho_t^{\text{FR}}) + V_* \big\rangle_{\rho_t^{\text{FR}}} \big) \rho_t^{\text{FR}}(x).
\end{align}
We henceforth omit the superscript $\text{FR}$ for the Fisher-Rao gradient flow of the KL divergence unless the notation becomes ambiguous. 
For short-hand, we make use of the abbreviation FR-GF for Fisher-Rao gradient flows. 

The FR-GF may be simulated using a system of weighted particles (see \cref{sec:details_sim}). Unlike for the  W-GF, in this case the positions of the particles are fixed; only the weights change over time. Hence, to simulate the FR-GF one is forced to grid the underlying space $\R^d$. This is feasible only for small dimensions $d$. Consequently, FR-GFs cannot be simulated in high dimensions, which makes them impractical for sampling applications. 

\subsubsection{Wasserstein-Fisher-Rao geometry gradient flow} \label{subsec:WFR}
The Wasserstein-Fisher-Rao distance between probability measures arises as a combination of the Wasserstein and the Fisher-Rao distances \citep{chizat2018unbalanced,chizat2015aninterpolating,kondratyev2016anew,liero2016optimal,liero2018optimal}. It is defined as
\begin{align*}
    \mathrm{WFR}^2(\rho_1,\rho_1) \defeq \inf_{(\rho_t,v_t,r_t)} \int_0^1 \int (  \|v_t(x)\|^2 + r_t(x)^2) \rho_t(x) \dd x \dd t \quad \text{s.t.} \quad \partial_t  \rho_t + \nabla \cdot (\rho_t v_t) = r_t \rho_t \,,
\end{align*}
where, for each $t \in [0,1]$, the triple $(\rho_t,v_t,r_t)$ lives in  $\cP(\Rd) \times L^2(\Rd)^d \times L^2(\Rd),$ and they simultaneously satisfy the constraint equation, which has endpoints $\rho_0$ and $\rho_1$, as well. 
Similarly, the Wasserstein-Fisher-Rao gradient flow of the KL divergence is the solution of PDE that incorporates the terms in the Wasserstein and Fisher-Rao gradient flows (\cref{eq:W_flow} and \cref{eq:FR_KL}):
\begin{align} \label{eq:WFR_flow}
    \partial_t \rho_t^{\text{WFR}} = \nabla \cdot \left(\rho_t^{\text{WFR}} (\nabla \log \rho_t^{\text{WFR}} + \nabla V_*) \right) -\big( \log(\rho_t^{\text{WFR}}) + V_* - \big\langle \log(\rho_t^{\text{WFR}}) + V_* \big\rangle_{\rho_t^{\text{WFR}}} \big) \rho_t^{\text{WFR}}
\end{align}
Similar to the other geometries, we write WFR-GF as shorthand for Wasserstein-Fisher-Rao gradient flow At the particle level, WFR-GFs are able to capture both \emph{transport} and \emph{weight updates}, which is why they enjoy a convergence rate that at least matches the better rate between W- and FR-GFs (recall \cref{eq: wfr_conv}), and is clearly superior in practice {in some instances.} Hence, any improvement in the convergence analysis of either W- or FR-GFs translates to improving our understanding of WFR-GFs.

\subsection{Simulated annealing dynamics}\label{sec: simulated_annealing}
Simulated annealing is a technique seen in several works when attempting to either optimize a function or sample from a multimodal probability distribution, and has a long history \citep{pincus1970amontecarlo,kirkpatrick1983optimization}, and plays a crucial role in our analysis. In what follows, we introduce the annealing path with linear scaling, and conclude with a proposition.

Consider the time-dependent measure $(\mu_\tau)_{\tau \in [0,1]}$ corresponding to the annealing path, with \textit{linear scaling}, initialized at the measure $\mu_0 = \rho_0 \propto e^{-V_0}$. By definition, $\mu_\tau$ admits the density
\begin{align} \label{eq:annealed_dyn}
    \mu_{\tau}(x) = \frac{e^{-\tau (V_*(x)-V_0(x)) - V_0(x)}}{Z_{\tau}}, \quad Z_{\tau} = \int_{\R^d} e^{-\tau (V_*(x)-V_0(x)) - V_0(x)} \, dx,
\end{align}
for $\tau \in [0,1]$. Note that indeed, $\mu_1 = \pi$. To this end, it will be convenient to rewrite \cref{eq:annealed_dyn} in terms of the log-density of $\mu_\tau$. Remark that

\begin{align}
    \log(\mu_\tau(x)) = -\tau (V_*(x) - V_0(x)) - V_0(x) - \log Z_\tau\,.
\end{align}
One can check that the pointwise derivative of the density $\mu_\tau$ (with respect to $\tau$) is
\begin{align}\label{eq:rho_star_general}
    \partial_\tau \mu_\tau (x) = - (V_*(x) - V_0(x) -  \langle V_* - V_0 \rangle_{\mu_\tau})\mu_\tau(x)\,.
\end{align}

From this, we obtain that
\begin{align}
\begin{split}
    &\log(\mu_\tau(x)) + E(x) - \big\langle \log(\mu_\tau) + V_* \big\rangle_{\mu_\tau} \\ &= -\tau (V_*(x) - V_0(x)) - V_0(x) - \langle - \tau (V_* - V_0) - V_0 + V_* \rangle_{\mu_{\tau}}  + V_*(x) - \langle V_* \rangle_{\mu_{\tau}}
    \\ &= -\tau (V_*(x) - V_0(x)) - V_0(x) + V_*(x) - \langle -\tau (V_* - V_0) - V_0 + V_* \rangle_{\mu_{\tau}} 
    \\ &= (1-\tau) \big( V_*(x) - V_0(x) \big) - (1-\tau) \langle V_* - V_0 \rangle_{\mu_\tau}
    \\ &= (1-\tau) \big( V_*(x) - V_0(x) - \langle V_* - V_0 \rangle_{\mu_\tau} \big)\,.
\end{split}
\end{align}
Note that in the first equality we used that the log-partition is a constant and gets cancelled. Consequently, \cref{eq:rho_star_general} can be rewritten, for $\tau \in (0,1)$, as
\begin{align} \label{eq:rho_star_general2}
    \partial_\tau \mu_\tau(x) = - \frac{1}{1-\tau} \big(\log(\mu_\tau(x)) + V_*(x) - \big\langle \log(\mu_\tau) + V_* \big\rangle_{\mu_\tau} \big)\mu_\tau(x).
\end{align}

A first observation is that that the linear schedule $\tau$ in the exponent of \cref{eq:annealed_dyn} results in dynamics that resemble the Fisher-Rao gradient flow of the KL divergence, up to a reparameterization that can be made explicit. Indeed, if one compares \cref{eq:rho_star_general2} with \cref{eq:FR_KL}, the only difference is the factor $\frac{1}{1-\tau}$ in the right-hand side of \cref{eq:rho_star_general2}. Since the solution of the Fisher-Rao gradient flow of the KL divergence is unique (see \cref{prop:uniqueness} in \cref{sec: remaining_proofs}), an appropriate time reparameterization of the annealed dynamics (\ref{eq:annealed_dyn}) will yield the solution  (\ref{eq:FR_KL}). We summarize this observation in the following proposition, which we were unable to find a citation for in the literature.
\begin{proposition} \label{prop:FR_explicit}
    Let $(\mu_{\tau})_{\tau \in [0,1]}$ be as defined in \cref{eq:annealed_dyn}. The Fisher-Rao gradient flow $(\rho_t)_{t \geq 0}$ of $\kl{\rho}{\pi}$ (i.e. solving \cref{eq:FR_KL}) is given by $\rho_t = \mu_{1-e^{-t}}$.
\end{proposition}
\begin{proof}
If we write $t$ as a function of $\tau$, we have that
\begin{align} \label{eq:t_of_tau}
    \partial_{\tau} \rho_{t(\tau)} = \partial_{t} \rho_{t(\tau)} \frac{dt}{d\tau}(\tau) = - \frac{dt}{d\tau}(\tau) \big( \log(\rho_{t(\tau)}(x)) + E(x) - \big\langle \log(\rho_{t(\tau)}) + E \big\rangle_{\rho_{t(\tau)}} \big)\rho_{t(\tau)}(x).
\end{align}
Identifying $\rho_{t(\tau)}$ with $\rho_{\tau}$, and establishing a direct comparison with \cref{eq:rho_star_general2}, we obtain that for \cref{eq:t_of_tau} to hold, $t(\tau)$ must fulfill $\frac{dt}{d\tau}(\tau) = \frac{1}{1-\tau}$.
With the initial condition that $\tau(0)=0$, this differential equation has the following unique solution:
\begin{align}
    t(\tau) = \int_{0}^{\tau} \frac{1}{1-s} \, ds = -\log (1-\tau)\,.
\end{align}
That is, we have that $t(\tau) = -\log (1-\tau)$, or equivalently, $\tau(t) = 1-e^{-t}$.    
\end{proof}

\subsection{Cumulants and their power series}\label{sec: cumulants}
Our core argument hinges on observing a relation between the above gradient flows and their connection to cumulants of a random variable. Recall that for a random variable $Y$, its \textit{cumulant-generating function} to be $K_Y(z) = \log \mathbb{E}[e^{Yz}]$. The $n^{\text{th}}$ cumulant $\kappa_n$ of the random variable $Y$ is defined as the $n^{\text{th}}$ derivative of $K_Y$ evaluated at $z=0$, that is, $\kappa_n = K^{(n)}_Y(0)$. Similar to moment-generating functions, if $K_Y(z)$ is finite in some neighborhood of $z \in (-\epsilon_0,\epsilon_0)$, then it holds that $K_Y$ is smooth (in fact, holomorphic) (see e.g. \cite[Section II.12.8]{ShiryaevProbability}. Moreover, $K_Y(z)$ admits the following infinite series expansion
\begin{align*}
    K_Y(z) = \sum_{n \geq 1} \frac{\kappa_n}{n!} z^n\,. 
\end{align*}
In particular, one can easily check that $\kappa_1 = \E[Y]$ and $\kappa_2 = \text{Var}(Y)$.

\section{Main result}\label{sec: main}
The goal of this section is to prove our main result, which is an explicit expansion of the KL divergence in terms of log-cumulants of the random variable $\log\frac{\rho_0(X)}{\pi(X)}$ where $X \sim \pi$. We make the following assumptions throughout, and we will make their uses explicit when necessary.
\begin{description}
\item \textbf{(A1)} $V_* \in L_1(\pi)$,
\item \textbf{(A2)} There exists $\alpha \in \R_+$, such that $\inf_x \frac{\rho_0(x)}{\pi(x)^{1+\alpha}} > 0$.
\end{description}

Assumption \textbf{\textbf{(A1)}} ensures that $\pi$ has finite differential entropy, and is a relatively weak condition.
\textbf{(A2)} asks that at least some mass is initially placed along the support of $\pi$.
\textbf{(A2)} is, however, a much weaker assumption that what is currently used in the literature. To be precise, \cite{lu2019accelerating,lu2022birth} assume a particular case of \textbf{(A2)}, namely
\begin{description}
    \item \textbf{(B)} There exists $M > 0$ such that $\inf_x \frac{\rho_0(x)}{\pi(x)} \geq e^{-M}$.
\end{description}
This is essentially the same as \textbf{(A2)}, though $\alpha$ is constrained to be 0, and a precise lower bound on the infimum is needed. Note that \textbf{(A2)} is weaker the larger $\alpha$ is, as $\pi(x)^{1+\alpha}$ decreases faster. As a comparison, if $\rho_0$ and $\pi$ are Gaussians, \textbf{(A2)} covers the setting where both have arbitrary means and covariances, while constraining $\alpha = 0$ only covers the cases in which the covariance matrix of $\rho_0$ is strictly larger than the one of $\pi$ in the positive definite order.

The following theorem is our main contribution. While here we have stated an asymptotic expression, in fact a more general expression is available as an infinite power series under the same assumptions, and appears explicitly in the proof.

\begin{theorem}\label{thm: main}
Suppose \textbf{(A1)} and \textbf{(A2)} hold. Then for $t$ large enough and any $q \in (1,\infty)$,
\begin{align} \label{eq:kl_renyi_approx}
&\kl{\rho_t}{\pi} = \frac{\kappa_2}{2}e^{-2t} + O(e^{-3t})\,, \quad \text{ and } \quad \cR_q(\rho_t\| \pi) = \frac{q \kappa_2}{2}e^{-2t}   + O_q(e^{-3t})\,,
\end{align}
where $\kappa_2 = \text{Var}_\pi\left( \log \frac{\rho_0}{\pi}\right)$.
\end{theorem}

\begin{remark}
The coefficient $\kappa_2$ is nothing more than the variance under $\pi$ of the first-variation of the KL divergence at $\rho_0$ (recall \cref{eq: first_var_kl}).
\end{remark}

\subsection{Proof}

Henceforth, we will always write
\begin{align}\label{eq: def_y}
Y \defeq \log\frac{\rho_0(X)}{\pi(X)} \text{ where } X \sim \pi\,.
\end{align}
 The first step in our proof is to represent these divergences as a function of the cumulants of $Y$, which is possible due to the aforementioned time-reparameterization of the FR flow. 

\begin{proposition}\label{prop: diff_rep}
    Let $\pi \propto e^{-V_*}$ and $\rho_0 \propto e^{-V_0}$ be probability measures on $\Rd$, and let $Y$ be as in \cref{eq: def_y}. Let $(\mu_\tau)_{\tau \in [0,1]}$ be follow the simulated annealing dynamics from \cref{eq:annealed_dyn}. It holds that
    \begin{align} \label{eq:KL_K_Y}
        &\kl{\mu_\tau}{\pi} = (1-\tau)K_Y'(1-\tau) - K_Y(1-\tau)\,, \\
        & \cR_q(\mu_\tau\|\pi)  = \frac{1}{q-1}K_Y(q(1-\tau)) - \frac{q}{q-1}K_Y(1-\tau)\,.
    \end{align}
\end{proposition}
\begin{proof}
We first identify the following relationship, which arises from a simple manipulation of \cref{eq:annealed_dyn}
\begin{align}\label{eq: Z_tau_Z_1}
    \log Z_\tau = K_Y(1-\tau) + \log Z_1\,.
\end{align}    
Using this expression, we can expand the KL divergence between $\mu_\tau$ and $\pi$ as follows:
\begin{align*}
    \kl{\mu_\tau}{\pi} &= \int \log \frac{\mu_\tau}{\pi}\mu_\tau = \int \log \left( \frac{e^{-\tau(V_* - V_0) - V_0} Z_\tau^{-1} }{e^{-V_*} Z_1^{-1}}\right) \dd \mu_\tau \\
    &= \log Z_1 - \log Z_\tau + (1-\tau)\langle V_* - V_0 \rangle_{\mu_\tau} \\
    &= (1-\tau) \langle V_* - V_0 \rangle_{\mu_\tau} - K_Y(1-\tau)\,.
\end{align*}
Another fact about cumulant generating functions that we can exploit is the following differential relationship
\begin{align}
    - \langle V_* - V_0 \rangle_{\mu_\tau} = \frac{\dd}{\dd \tau } Z_\tau = -K_Y'(1-\tau)\,.
\end{align}
Altogether, this gives
\begin{align}\label{eq: kl_with_tau}
    \kl{\mu_\tau}{\pi} = (1-\tau)K_Y'(1-\tau) - K_Y(1-\tau)\,.
\end{align}
The general $q$-R{\'e}nyi case is deferred to the appendix, where the computation is similar.
\end{proof}

The following lemma uses both \textbf{(A1)} and \textbf{(A2)} to establish that $K_Y(z)$ is finite in some neighborhood of $z \in B_{\epsilon_0}(0)$, which implies that $K_Y$ admits the series expansion we will require in the sequel.
The proof is deferred to the appendix.

\begin{proposition}\label{prop: k_y is finite}Suppose \textbf{(A1)} and \textbf{(A2)} are satisfied. Then there exists some constant $\epsilon_0 > 0$ such that the cumulant generating function of $Y$, $K_Y(z) = \log \E[e^{Yz}]$ is finite on some neighborhood of $z \in B_{\epsilon_0}(0)$.  Moreover, inside this neighborhood, $K_Y(z)$ is holomorphic and we have the series expansion
\begin{align}
K_Y(z) = \sum_{n \geq 1}\frac{\kappa_n}{n!} z^n\,.
\end{align}
\end{proposition}

We conclude with the proof of our main result.
\begin{proof}[Proof of \Cref{thm: main}]
We begin with the expression of the KL divergence. Note that since $K_Y(z)$ is smooth for $z$ sufficiently close to the origin, it holds that
\begin{align*}
K_Y'(z) = \sum_{n\geq 1} \frac{\kappa_n}{(n-1)!}z^{n-1}\,.
\end{align*}
Using the parameterization of \cref{eq: kl_with_tau} and the series expansion for $K_Y'(1-\tau)$, our expression for $\kl{\mu_\tau}{\pi}$ reads
\begin{align*}
\kl{\mu_\tau}{\pi} &= (1-\tau) \sum_{n\geq 1} \frac{\kappa_n}{(n-1)!}(1-\tau)^{n-1} - \sum_{n\geq 1}\frac{\kappa_n}{n!}(1-\tau)^n \\
&= \sum_{n\geq 1} \kappa_n \left(\frac{n}{n!} - \frac{1}{n!}\right) (1-\tau)^n \\
&= \sum_{n\geq 2} \frac{\kappa_n}{n(n-2)!} (1-\tau)^n\,.
\end{align*}
Expanding the relation and replacing $\tau(t) = 1-e^{-t}$ gives
\begin{align*}
\kl{\rho_t}{\pi} = \frac{\kappa_2}{2}e^{-2t} + \sum_{n\geq 3}\frac{\kappa_n}{n(n-2)!}e^{-nt}\,.
\end{align*}

We now do the same manipulations for $\cR_q(\mu_\tau\|\pi)$.
\begin{align*}
\cR_q(\mu_\tau \| \pi) &= \frac{1}{q-1}\sum_{n\geq 1}\frac{\kappa_n}{n!}(q(1-\tau))^n - \frac{q}{q-1}\sum_{n \geq 1} \frac{\kappa_n}{n!}(1-\tau)^n \\
&= \frac{1}{q-1}\left( \frac{\kappa_1}q(1-\tau) + \sum_{n \geq 2}q^n \frac{\kappa_n}{n!} (1-\tau)^n \right) - \frac{q}{q-1} \left( \kappa_1 (1-\tau) + \sum_{n \geq 2} \frac{\kappa_n}{n!} (1-\tau)^n \right) \\
&= \sum_{n\geq 2} \frac{q^n - q}{q-1} \frac{\kappa_n}{n!}(1-\tau)^n\,.
\end{align*} 
Substituting $\tau(t) = 1 - e^{-t}$ and expanding out the first term yields
\begin{align*}
\cR_q(\rho_t\|\pi) = q \frac{\kappa_2}{2}e^{-2t} + \sum_{n\geq 3} \frac{q^n - q}{q-1} \frac{\kappa_n}{n!} e^{-nt}\,.
\end{align*}

Our proof concludes by taking the limit $t \to \infty$, which we fully justify in the appendix (\cref{lem: o_limit}).
\end{proof}

\section{Numerical simulations} \label{sec:simulations}
We present simple numerical simulations that demonstrates our asymptotic convergence rate of the KL divergence the FR gradient flows, as well as a comparison with the WFR- and W-GFs. We consider two target distributions over the set $[-\pi,\pi)$, each with two initializations:
\begin{enumerate}
    \item Target distribution $\pi_1$: We set $\pi_1 \propto e^{-V_1}$ with $V_1(x) = 2.5\cos(2x) + 0.5\sin(x)$. This distribution has two modes with different weights and has been studied previously by \cite{lu2019accelerating}. We consider two initial distributions: 
    \begin{enumerate}[label=(\alph*)]
    \item $\pi_a \propto e^{-V_a}$ with $V_a = - V_1$, which has two modes in locations where $\pi$ has little mass.
    \item $\pi_b \propto e^{-V_b}$ with $V_b = 2.5\cos(2x)$, which has two modes in almost the same positions as $\pi$, but with equal weight.
    \end{enumerate}
    \item Target distribution $\pi_2$: We set $\pi_2 \propto e^{-V_2}$ with $V_2(x) = -6\cos(x)$. This distribution has one mode. We consider two initial distributions: 
    \begin{enumerate}[label=(\alph*)]
    \item[(c)] $\pi_c \propto e^{-V_c}$ with $V_c = - V_2$, which has one mode in a location where $\pi$ has little mass.
    \item[(d)] $\pi_d \propto e^{-V_d}$ with $V_d = 0$, which is the uniform distribution.
    \end{enumerate}
\end{enumerate}

\begin{figure}[h!]
    \centering
    \begin{tabular}{cc}
    \includegraphics[width=0.39\textwidth]{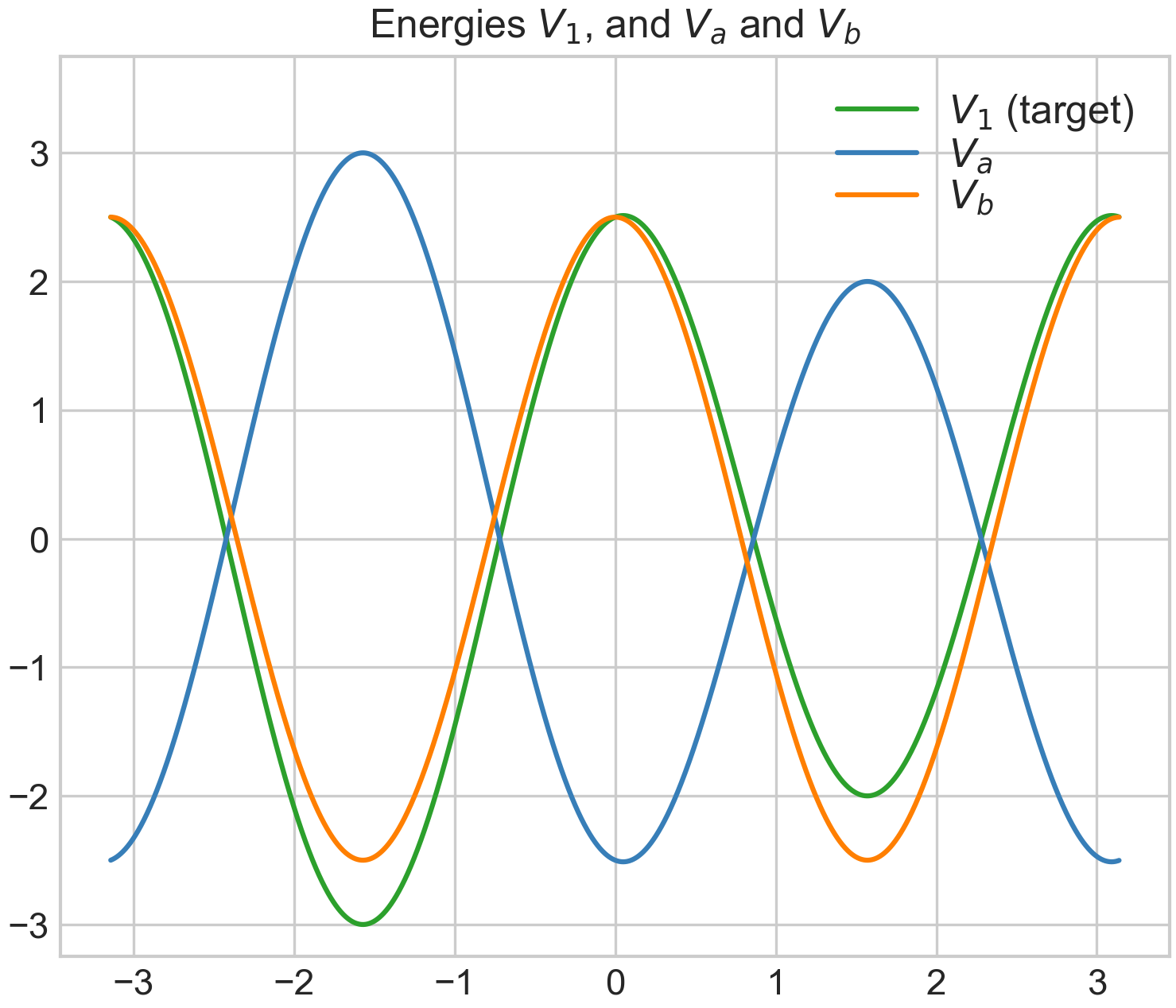} & \includegraphics[width=0.39\textwidth]{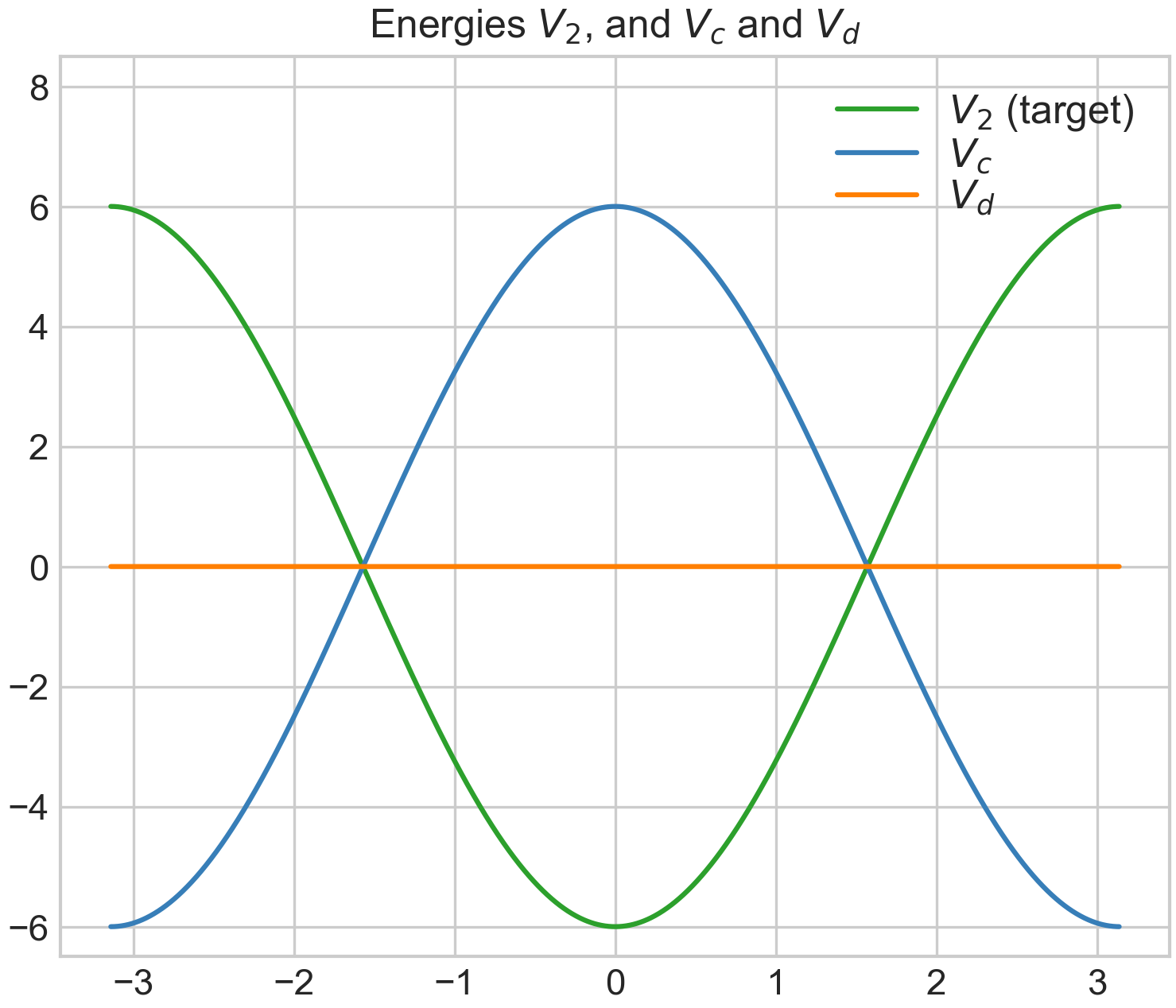}
    \end{tabular}
    \caption{Energies of the target and initial distributions.}
    \label{fig:energies}
\end{figure}

\cref{fig:energies} shows the target energies $V_1$, $V_2$ and the initial energies $V_a$, $V_b$, $V_c$, $V_d$ introduced above. 
\cref{fig:rates} shows the evolution of the KL divergence along the FR, WFR and W gradient flows. It also contains plots of the dominant term $\frac{\kappa_2}{2}e^{-2t}$ of the approximation of the KL divergence decay for FR flows (see \cref{thm: main}), displayed as dotted lines. \cref{table:slopes} shows the slopes of each curve from \cref{fig:rates}, at large times (see \cref{sec:details_sim} for details on the computation of slopes). 

\begin{figure}[h]
    \centering
    \begin{tabular}{cc}
    \includegraphics[width=0.42\textwidth]{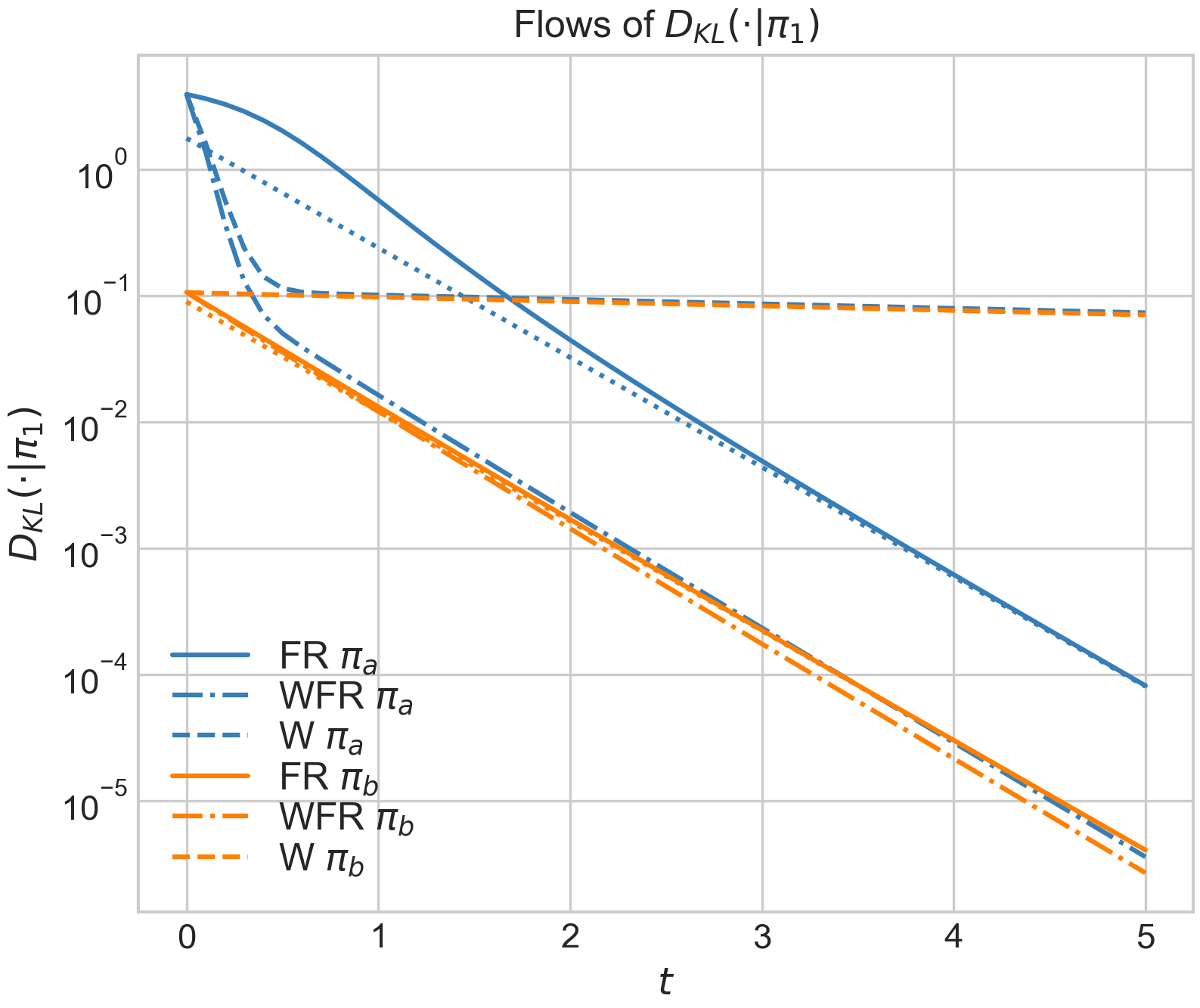} & 
    \includegraphics[width=0.42\textwidth]{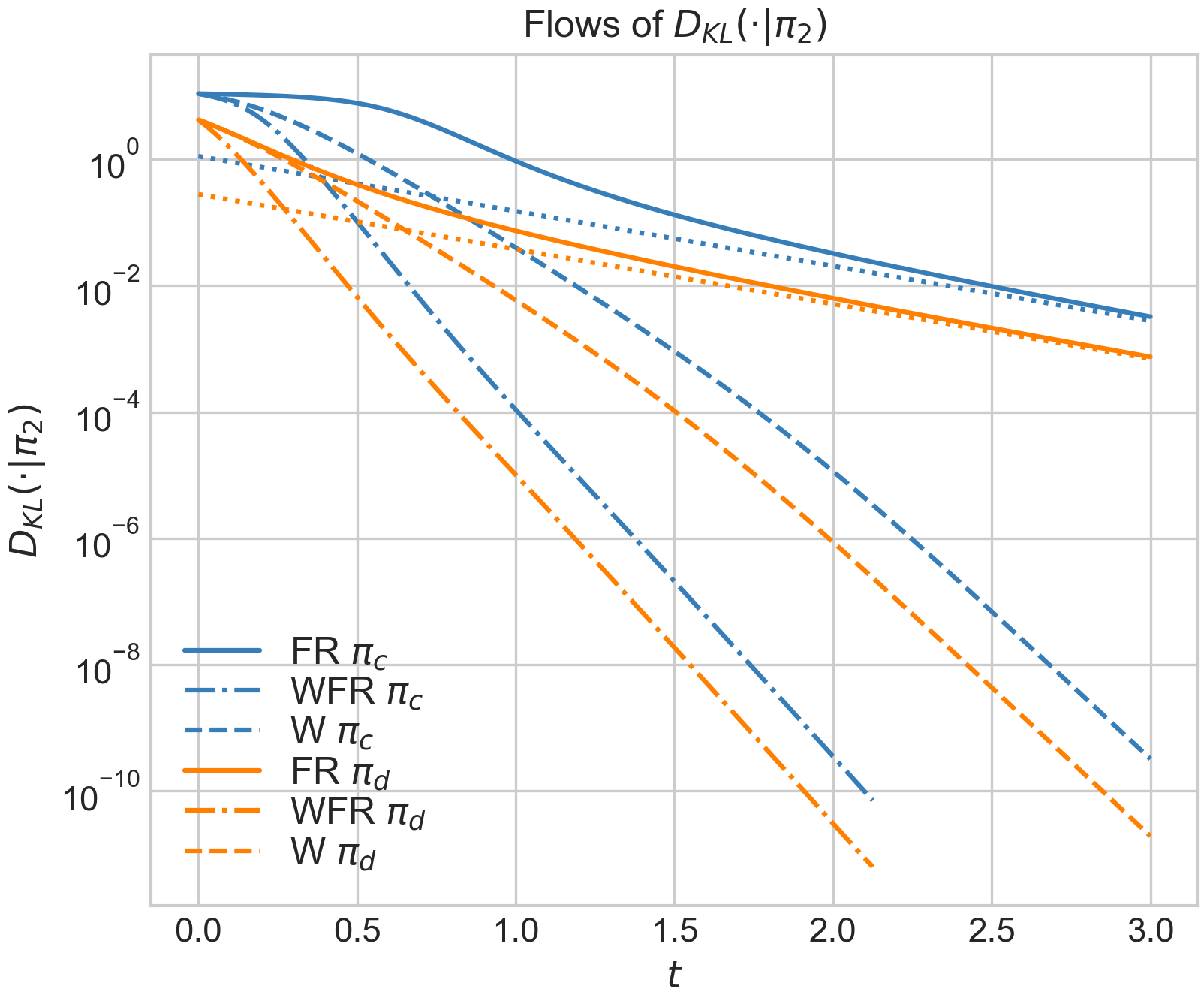} 
    \end{tabular}
    \caption{Evolution of the KL divergence with respect to $\pi_1$ (\emph{left}) and $\pi_2$ along their respective FR (\emph{solid lines}), WFR (\emph{dash-dotted lines}) and W (\emph{dashed lines}) gradient flows. Each plot contains flows initialized at two probability measures: in the left plot these are $\pi_a$ (\emph{blue}, top curves at $t=0$) and $\pi_b$ (\emph{orange}); in the right plot, $\pi_c$ (\emph{blue}, top curves at $t=0$) and $\pi_d$ (\emph{orange}). The \emph{dotted} lines show the curves $\frac{\kappa_2}{2}e^{-2t}$ (for the appropriate values $\kappa_2$), introduced in \autoref{thm: main}.}
    \label{fig:rates}
\end{figure}

Some observations are in order:
\begin{itemize}
    \item As predicted by \cref{thm: main}, the curves $\kl{\rho_t^{\text{FR}}}{\pi}$ approach the curves $\frac{\kappa_2}{2}e^{-2t}$ as $t$ grows. 
    \item For $\pi_1$, the curves $\kl{\rho_t^{\text{FR}}}{\pi}$ and $\kl{\rho_t^{\text{WFR}}}{\pi}$ initialized at $\pi_b$ are very close for small times. The reason is that $\nabla V_1$ and $\nabla V_b$ are very close in the regions where $\pi_1$ and $\pi_b$ have most of the mass. Consequently, the term $\nabla \cdot \left(\rho_t^{\text{WFR}} (\nabla \log \rho_t^{\text{WFR}} + \nabla V_1) \right)$, which is the difference between the FR and the WFR PDEs, is small at initialization.
    \item The curves $\kl{\rho_t^{\text{W}}}{\pi}$ behave very differently for $\pi_1$ and $\pi_2$ (see \cref{table:slopes}). Indeed, since $\pi_1$ is bimodal $C_{\texttt{LSI}}(\pi_1)$ is quite large (thus convergence is slow), whereas $\pi_2$ is unimodal, with a much smaller log-Sobolev constant. 
    \item The curves $\kl{\rho_t^{\text{WFR}}}{\pi}$ also behave differently for both target distributions. For $\pi_1$, it decays only slightly faster than $\kl{\rho_t^{\text{FR}}}{\pi}$, while for $\pi_2$ it goes down much faster than both $\kl{\rho_t^{\text{FR}}}{\pi}$ and $\kl{\rho_t^{\text{WFR}}}{\pi}$. Interestingly, looking at \cref{table:slopes} we observe that the asymptotic slopes of the WFR are very close to the sum of the slopes for FR and W. This seems to indicate that at large times, the KL divergence decays like $e^{- 2t - \frac{2t}{C_{\texttt{LSI}}}}$, i.e. that the W and FR terms act more or less independently.
\end{itemize}

\begin{table}[h!]
\centering
\begin{tabular}{l|ll|ll}
 & \multicolumn{2}{c}{Target $\pi_1$}  &  \multicolumn{2}{c}{Target $\pi_2$} \\
 & Init. $\pi_a$ & Init. $\pi_b$ & Init. $\pi_c$ & Init. $\pi_d$ \\ \hline
FR & -2.0016 & -2.0002 & -2.0028 & -2.0014 \\
WFR & -2.0771 & -2.0759 & -12.8190 & -12.8632 \\
W & -0.0811 & -0.0811 & -10.7784 & -10.8538
\end{tabular}
\caption{Large-time slopes of the KL divergence vs. time curves in a semi-logarithmic plot (\cref{fig:rates}), for the three flows. See \cref{sec:details_sim} for details on the computation of the slopes.}
\label{table:slopes}
\end{table}

\section{Conclusion}
In this work, using a relatively simple proof technique, we showed that the Kullback-Leibler divergence along its Fisher-Rao gradient flow $(\rho_t^{\text{FR}})_{t \geq 0}$ can be written as a power-series expansion, resulting in a tight asymptotic convergence rate for large times. 
A similar expansion holds for $\cR_q(\rho_t^{\text{FR}}\|\pi)$, where $\cR_q$ is any $q$-R{\'e}nyi divergence.
Our findings were verified with simple numerical experiments, where we also simulated Wasserstein and Wasserstein-Fisher-Rao gradient flows.
Our simulations indicated that, in some cases, the convergence rate of the WFR gradient flow scales like $e^{- (2+ (2/C_{\texttt{LSI}}))t}$, an observation that we hope can be made precise in future work.
A second direction is to extend our proof technique from the KL divergence to general Bregman divergences.

\subsubsection*{Acknowledgments}
The authors thank Joan Bruna, Jonathan Niles-Weed, Sinho Chewi, and Andre Wibisono for helpful discussions. CD acknowledges Meta AI Research as a funding source. AAP acknowledges NSF Award 1922658 and and Meta AI Research.

\bibliography{main}

\begin{thebibliography}{39}
\providecommand{\natexlab}[1]{#1}
\providecommand{\url}[1]{\texttt{#1}}
\expandafter\ifx\csname urlstyle\endcsname\relax
  \providecommand{\doi}[1]{doi: #1}\else
  \providecommand{\doi}{doi: \begingroup \urlstyle{rm}\Url}\fi

\bibitem[Ambrosio et~al.(2005)Ambrosio, Gigli, and
  Savar{\'e}]{ambrosio2005gradient}
Luigi Ambrosio, Nicola Gigli, and Giuseppe Savar{\'e}.
\newblock \emph{Gradient flows: in metric spaces and in the space of
  probability measures}.
\newblock Springer Science \& Business Media, 2005.

\bibitem[Arnold et~al.(2000)Arnold, Markowich, and
  Unterreiter]{arnold2000onconvex}
Anton Arnold, Peter Markowich, and Andreas Unterreiter.
\newblock On convex {Sobolev} inequalities and the rate of convergence to
  equilibrium for {Fokker-Planck} type equations.
\newblock \emph{Communications in Partial Differential Equations}, 26, 05 2000.

\bibitem[Benamou \& Brenier(2000)Benamou and Brenier]{benamou2000computational}
Jean-David Benamou and Yann Brenier.
\newblock A computational fluid mechanics solution to the {Monge-Kantorovich}
  mass transfer problem.
\newblock \emph{Numerische Mathematik}, 84\penalty0 (3):\penalty0 375--393,
  2000.

\bibitem[Bogachev(2007)]{bogachev2007measure}
V.I. Bogachev.
\newblock \emph{Measure Theory}.
\newblock Number~1 in Measure Theory. Springer Berlin Heidelberg, 2007.

\bibitem[Chewi(2022)]{sinhobook}
Sinho Chewi.
\newblock \emph{Log-concave sampling}.
\newblock 2022.

\bibitem[Chewi et~al.(2022)Chewi, Erdogdu, Li, Shen, and
  Zhang]{chewi2022analysis}
Sinho Chewi, Murat~A Erdogdu, Mufan Li, Ruoqi Shen, and Shunshi Zhang.
\newblock Analysis of {Langevin Monte Carlo} from {Poincare to Log-Sobolev}.
\newblock In \emph{Proceedings of Thirty Fifth Conference on Learning Theory},
  volume 178 of \emph{Proceedings of Machine Learning Research}. PMLR, 02--05
  Jul 2022.

\bibitem[Chizat(2022)]{chizat2022sparse}
Lenaic Chizat.
\newblock Sparse optimization on measures with over-parameterized gradient
  descent.
\newblock \emph{Mathematical Programming}, 194\penalty0 (1-2):\penalty0
  487--532, 2022.

\bibitem[Chizat et~al.(2015)Chizat, Schmitzer, Peyré, and
  Vialard]{chizat2015aninterpolating}
Lenaic Chizat, Bernhard Schmitzer, Gabriel Peyré, and François-Xavier
  Vialard.
\newblock An interpolating distance between optimal transport and {Fisher-Rao}.
\newblock \emph{Foundations of Computational Mathematics}, 18, 06 2015.

\bibitem[Chizat et~al.(2018)Chizat, Peyré, Schmitzer, and
  Vialard]{chizat2018unbalanced}
Lénaïc Chizat, Gabriel Peyré, Bernhard Schmitzer, and François-Xavier
  Vialard.
\newblock Unbalanced optimal transport: {dynamic and Kantorovich formulations}.
\newblock \emph{Journal of Functional Analysis}, 274\penalty0 (11):\penalty0
  3090--3123, 2018.

\bibitem[Dalalyan \& Tsybakov(2012)Dalalyan and Tsybakov]{dalalyan2012sparse}
A.S. Dalalyan and A.B. Tsybakov.
\newblock Sparse regression learning by aggregation and {Langevin Monte-Carlo}.
\newblock \emph{Journal of Computer and System Sciences}, 78\penalty0
  (5):\penalty0 1423--1443, 2012.

\bibitem[Durmus et~al.(2021)Durmus, Majewski, and
  Miasojedow]{durmus2021analysis}
Alain Durmus, Szymon Majewski, and B\l{}a\.{z}ej Miasojedow.
\newblock Analysis of {Langevin Monte Carlo} via convex optimization.
\newblock \emph{J. Mach. Learn. Res.}, 20\penalty0 (1):\penalty0 2666–2711,
  2021.

\bibitem[Gallou{\"e}t \& Monsaingeon(2017)Gallou{\"e}t and
  Monsaingeon]{gallouet2017jko}
Thomas~O Gallou{\"e}t and Leonard Monsaingeon.
\newblock A {JKO} splitting scheme for {K}antorovich--{F}isher--{R}ao gradient
  flows.
\newblock \emph{SIAM Journal on Mathematical Analysis}, 49\penalty0
  (2):\penalty0 1100--1130, 2017.

\bibitem[Gelman et~al.(1995)Gelman, Carlin, Stern, and
  Rubin]{gelman1995bayesian}
Andrew Gelman, John~B Carlin, Hal~S Stern, and Donald~B Rubin.
\newblock \emph{Bayesian data analysis}.
\newblock Chapman and Hall/CRC, 1995.

\bibitem[Gross(1975)]{gross1975logarithmic}
Leonard Gross.
\newblock Logarithmic {Sobolev} inequalities.
\newblock \emph{American Journal of Mathematics}, 97\penalty0 (4):\penalty0
  1061--1083, 1975.

\bibitem[Hellinger(1909)]{hellinger1909neue}
E.~Hellinger.
\newblock {Neue Begründung der Theorie quadratischer Formen von
  unendlichvielen Veränderlichen}.
\newblock \emph{Journal für die reine und angewandte Mathematik}, \penalty0
  (136):\penalty0 210--271, 1909.

\bibitem[Holley \& Stroock(1987)Holley and Stroock]{Holley1987}
Richard Holley and Daniel Stroock.
\newblock Logarithmic {Sobolev} inequalities and stochastic {Ising} models.
\newblock \emph{Journal of Statistical Physics}, 46\penalty0 (5):\penalty0
  1159–1194, Mar 1987.
\newblock ISSN 1572-9613.

\bibitem[Johannes \& Polson(2010)Johannes and Polson]{johannes2010mcmc}
Michael Johannes and Nicholas Polson.
\newblock {MCMC} methods for continuous-time financial econometrics.
\newblock In \emph{Handbook of Financial Econometrics: Applications}, pp.\
  1--72. Elsevier, 2010.

\bibitem[Jordan et~al.(1998)Jordan, Kinderlehrer, and
  Otto]{jordan1998variational}
Richard Jordan, David Kinderlehrer, and Felix Otto.
\newblock The variational formulation of the {F}okker--{P}lanck equation.
\newblock \emph{SIAM journal on mathematical analysis}, 29\penalty0
  (1):\penalty0 1--17, 1998.

\bibitem[Kakutani(1948)]{kakutani1948onequivalence}
Shizuo Kakutani.
\newblock On equivalence of infinite product measures.
\newblock \emph{Annals of Mathematics}, 49\penalty0 (1):\penalty0 214–--224,
  1948.

\bibitem[Kirkpatrick et~al.(1983)Kirkpatrick, Gelatt, and
  Vecchi]{kirkpatrick1983optimization}
S.~Kirkpatrick, C.~D. Gelatt, and M.~P. Vecchi.
\newblock Optimization by simulated annealing.
\newblock \emph{Science}, 220\penalty0 (4598):\penalty0 671--680, 1983.

\bibitem[Kobyzev et~al.(2020)Kobyzev, Prince, and
  Brubaker]{kobyzev2020normalizing}
Ivan Kobyzev, Simon~JD Prince, and Marcus~A Brubaker.
\newblock Normalizing flows: An introduction and review of current methods.
\newblock \emph{IEEE transactions on pattern analysis and machine
  intelligence}, 43\penalty0 (11):\penalty0 3964--3979, 2020.

\bibitem[Kondratyev et~al.(2016)Kondratyev, Monsaingeon, and
  Vorotnikov]{kondratyev2016anew}
Stanislav Kondratyev, L{\'e}onard Monsaingeon, and Dmitry Vorotnikov.
\newblock {A new optimal transport distance on the space of finite {Radon}
  measures}.
\newblock \emph{Advances in Differential Equations}, 21\penalty0
  (11/12):\penalty0 1117 -- 1164, 2016.

\bibitem[Lambert et~al.(2022)Lambert, Chewi, Bach, Bonnabel, and
  Rigollet]{lambert2022variational}
Marc Lambert, Sinho Chewi, Francis Bach, Silv{\`e}re Bonnabel, and Philippe
  Rigollet.
\newblock Variational inference via {W}asserstein gradient flows.
\newblock \emph{arXiv preprint arXiv:2205.15902}, 2022.

\bibitem[Liero et~al.(2016)Liero, Mielke, and Savar\'{e}]{liero2016optimal}
Matthias Liero, Alexander Mielke, and Giuseppe Savar\'{e}.
\newblock Optimal transport in competition with reaction: The
  {Hellinger--Kantorovich} distance and geodesic curves.
\newblock \emph{SIAM Journal on Mathematical Analysis}, 48\penalty0
  (4):\penalty0 2869--2911, 2016.

\bibitem[Liero et~al.(2018)Liero, Mielke, and Savaré]{liero2018optimal}
Matthias Liero, Alexander Mielke, and Giuseppe Savaré.
\newblock Optimal entropy-transport problems and a new {Hellinger-Kantorovich}
  distance between positive measures.
\newblock \emph{Inventiones mathematicae}, 211, 03 2018.

\bibitem[Lu et~al.(2019)Lu, Lu, and Nolen]{lu2019accelerating}
Yulong Lu, Jianfeng Lu, and James Nolen.
\newblock Accelerating {L}angevin sampling with birth-death.
\newblock \emph{arXiv preprint arXiv:1905.09863}, 2019.

\bibitem[Lu et~al.(2022)Lu, Slep{\v{c}}ev, and Wang]{lu2022birth}
Yulong Lu, Dejan Slep{\v{c}}ev, and Lihan Wang.
\newblock Birth-death dynamics for sampling: {G}lobal convergence,
  approximations and their asymptotics.
\newblock \emph{arXiv preprint arXiv:2211.00450}, 2022.

\bibitem[MacKay(2003)]{mackay2003information}
David~JC MacKay.
\newblock \emph{Information theory, inference and learning algorithms}.
\newblock Cambridge university press, 2003.

\bibitem[Markowich \& Villani(1999)Markowich and Villani]{Markowich99onthe}
P.~A. Markowich and C.~Villani.
\newblock On the trend to equilibrium for the {Fokker-Planck} equation: An
  interplay between physics and functional analysis.
\newblock In \emph{Physics and Functional Analysis, Matematica Contemporanea
  (SBM) 19}, pp.\  1--29, 1999.

\bibitem[Pincus(1970)]{pincus1970amontecarlo}
Martin Pincus.
\newblock A {Monte Carlo} method for the approximate solution of certain types
  of constrained optimization problems.
\newblock \emph{Operations Research}, 18\penalty0 (6):\penalty0 1225--1228,
  1970.

\bibitem[Robert et~al.(1999)Robert, Casella, and Casella]{robert1999monte}
Christian~P Robert, George Casella, and George Casella.
\newblock \emph{Monte {C}arlo statistical methods}, volume~2.
\newblock Springer, 1999.

\bibitem[Rotskoff et~al.(2019)Rotskoff, Jelassi, Bruna, and
  Vanden-Eijnden]{rotskoff2019global}
Grant Rotskoff, Samy Jelassi, Joan Bruna, and Eric Vanden-Eijnden.
\newblock Global convergence of neuron birth-death dynamics.
\newblock \emph{arXiv preprint arXiv:1902.01843}, 2019.

\bibitem[Shiryaev(1984)]{ShiryaevProbability}
Al'bert~Nikolaevich Shiryaev.
\newblock \emph{Probability}.
\newblock Graduate texts in mathematics ; 95. Springer-Verlag, New York, 1984.
\newblock ISBN 9781489900180.

\bibitem[Stam(1959)]{stam1959some}
A.J. Stam.
\newblock Some inequalities satisfied by the quantities of information of
  {Fisher} and {Shannon}.
\newblock \emph{Information and Control}, 2\penalty0 (2):\penalty0 101--112,
  1959.

\bibitem[Vempala \& Wibisono(2019)Vempala and Wibisono]{vempala2019rapid}
Santosh Vempala and Andre Wibisono.
\newblock Rapid convergence of the unadjusted {Langevin} algorithm:
  Isoperimetry suffices.
\newblock In \emph{Advances in Neural Information Processing Systems},
  volume~32. Curran Associates, Inc., 2019.

\bibitem[Villani(2008)]{villani2008optimal}
C.~Villani.
\newblock \emph{Optimal Transport: Old and New}.
\newblock Grundlehren der mathematischen Wissenschaften. Springer Berlin
  Heidelberg, 2008.

\bibitem[Von~Toussaint(2011)]{von2011bayesian}
Udo Von~Toussaint.
\newblock Bayesian inference in physics.
\newblock \emph{Reviews of Modern Physics}, 83\penalty0 (3):\penalty0 943,
  2011.

\bibitem[Wibisono(2018)]{wibisono2018sampling}
Andre Wibisono.
\newblock Sampling as optimization in the space of measures: The {Langevin}
  dynamics as a composite optimization problem.
\newblock In \emph{Conference on Learning Theory}, pp.\  2093--3027. PMLR,
  2018.

\bibitem[Yan et~al.(2023)Yan, Wang, and Rigollet]{yan2023learning}
Yuling Yan, Kaizheng Wang, and Philippe Rigollet.
\newblock Learning {G}aussian mixtures using the {W}asserstein-{F}isher-{R}ao
  gradient flow.
\newblock \emph{arXiv preprint arXiv:2301.01766}, 2023.

\end{thebibliography}
\bibliographystyle{tmlr}

\appendix
\section{Remaining proofs} \label{sec: remaining_proofs}

\begin{proposition}[Uniqueness of the Fisher-Rao gradient flow of the KL divergence] \label{prop:uniqueness}
Given a target potential $V_*$ and an initial measure $\rho_0$, the solution of \cref{eq:FR_KL} is unique.
\end{proposition}
\begin{proof}
    Consider the PDE
    \begin{align} \label{eq:PDE_unnormalized}
    \partial_t \mu_t(x) = -\big( \log(\mu_t(x)) + V_*(x)  
    \big) \mu_t(x), \qquad \mu_0 = \rho_0
    \end{align}
    Note that this is in fact an ODE for each point $x$, that we can rewrite as $\partial_t \log \mu_t(x) = -\big( \log(\mu_t(x)) + V_*(x) \big)$. The unique solution of this ODE with initial condition $\log \mu_0(x)$ is $\log \mu_t(x) = (\log \mu_0(x)-V_*(x)) e^{-t} + V_*(x)$. Thus, we conclude that \cref{eq:PDE_unnormalized} has a unique solution.

    Now, given a solution $\rho_t$ of \cref{eq:FR_KL} with initial condition $\rho_0$, define $\tilde{\rho}_t$ as
    \begin{align} \label{eq:map_rho_to_tilderho}
        \log \tilde{\rho}_t(x) = \log \rho_t(x) - \int_0^t \big\langle \log(\rho_s) + V_* \big\rangle_{\rho_s} \dd s.
    \end{align}
    Remark that $\tilde{\rho}_t$ is a solution of \cref{eq:PDE_unnormalized}, since
    \begin{align*}
        \partial_t \log \tilde{\rho}_t(x) &= \partial_t \log \rho_t(x) - \big\langle \log(\rho_t) + V_* \big\rangle_{\rho_t} = -\big( \log(\rho_t(x)) + V_*(x) - \big\langle \log(\rho_t) + V_* \big\rangle_{\rho_t} \big) - \big\langle \log(\rho_t) + V_* \big\rangle_{\rho_t} \\ &
        = -\big( \log(\rho_t(x)) + V_*(x) \big).
    \end{align*}
    Also, note that the map $(\rho_t)_{t\geq 0} \to (\tilde{\rho}_t)_{t\geq 0}$ defined by \cref{eq:map_rho_to_tilderho} is invertible, as $\rho_t(x) = \tilde{\rho_t}(x)/\int \tilde{\rho}_t(y) \dd y$. This follows from the fact that $\rho_t$ and $\tilde{\rho}_t$ are proportional to each other, and that $\rho_t$ integrates to 1.
    
    Finally, suppose that $\rho_t^a$ and $\rho_t^b$ are two solutions of \cref{eq:FR_KL} with initial condition $\rho_0$. Via the construction \cref{eq:map_rho_to_tilderho}, they yield solutions $\tilde{\rho}_t^a$ and $\tilde{\rho}_t^b$ of \cref{eq:PDE_unnormalized} with initial condition $\rho_0$. The uniqueness of the solution of \cref{eq:PDE_unnormalized} implies that $\tilde{\rho}_t^a = \tilde{\rho}_t^b$. Since the map $(\rho_t)_{t\geq 0} \to (\tilde{\rho}_t)_{t\geq 0}$ is invertible, we obtain that $\rho_t^a = \rho_t^b$, which concludes the proof
\end{proof}

\begin{proof}[Proof of \cref{prop: diff_rep} (Continued)]
    We perform similar manipulations as in the case with the KL divergence:
    \begin{align*}
        \cR_q(\mu_\tau\|\pi) &= \frac{1}{q-1} \log \int \frac{e^{-q \tau (V_* - V_0) - q V_0} (Z_\tau)^{-q} }{ e^{-q V_*} (Z_1)^{-q} } \dd \pi \\
        &= \frac{1}{q-1}\log \int e^{q(1-\tau)(V_* - V_0)} \left(\frac{Z_\tau}{Z_1}\right)^q \dd \pi \\
        &= \frac{1}{q-1}K_Y(1-\tau) - \frac{q}{q-1} (\log Z_\tau - \log Z_1) \\
        &= \frac{1}{q-1}K_Y(1-\tau) - \frac{q}{q-1} K_Y(1-\tau)\,,
    \end{align*}
    where in the last line we again used \cref{eq: Z_tau_Z_1}. This completes the proof.
\end{proof}

\begin{proof}[Proof of \Cref{prop: k_y is finite}]
By \textbf{(A1)}, the partition function $F(t) = \int_{\R^d} e^{-tV_*(x)} \, dx$ is differentiable at $t=1$. This is because $F'(t) = -\int V_*(x) \, d\pi(x)$. Hence, $F(t)$ is finite on an interval $(1-2\epsilon_1,1]$ for some $\epsilon_1$.

Note that the assumption \textbf{(A2)} can be written equivalently as $\xi := \inf_x \alpha V_*(x) - V_0(x) > -\infty$. We obtain that for all $\epsilon \in [0,\epsilon_1/\alpha)$,
    \begin{align}
    \begin{split}
        -\epsilon(V_*(x)-V_0(x)) - V_*(x) &= -\epsilon((1+\alpha)V_*(x)-V_0(x)) + (\epsilon \alpha - 1) V_*(x) \\ &\leq -\epsilon \xi +  (\epsilon \alpha - 1) V_*(x) \leq -\epsilon \xi +  (\epsilon_1 - 1) V_*(x)
    \end{split}
    \end{align}
    Equivalently,
    \begin{align}
        \exp(K_Y(-\epsilon)) = \int_{\R^d} e^{-\epsilon(V_*(x)-V_0(x)) - V_*(x)} \, dx \leq e^{-\epsilon \xi} \int_{\R^d} e^{-(1 - \epsilon_1) V_*(x)} \, dx = e^{-\epsilon \xi} F(1-\epsilon_1) < +\infty.
    \end{align}
    Also, for all $\epsilon \in [0,1)$, using the convexity of the exponential function we have that
    \begin{align}
        \exp(K_Y(\epsilon)) &= \int_{\R^d} e^{\epsilon(V_*(x)-V_0(x)) - V_*(x)} \, dx = \int_{\R^d} e^{-(1-\epsilon)V_*(x)-\epsilon V_0(x)} \, dx \\ &\leq \int_{\R^d} (1-\epsilon) e^{-V_*(x)} + \epsilon e^{-V_0(x)} \, dx = (1-\epsilon) Z_1 + \epsilon Z_0 < +\infty.
    \end{align}
    Hence, the cumulant-generating function $K_Y(t) = \log \E e^{tY}$ is finite on a neighborhood $(-\epsilon_0,\epsilon_0)$ with $\epsilon_0 = \min\{1, \epsilon_1/\alpha \}$. Applying Lemma \ref{lem:analytic_H}, we conclude that there exists $\epsilon > 0$ such that for $z \in B_{\epsilon}(0)$, we have that $K_Y(z) = \sum_{n = 1}^{+\infty} \frac{\kappa_n}{n!} z^n$.
\end{proof}

The following lemma, which we make explicit, is a well-known fact in probability theory. In short, since the moment-generating function is analytic in some neighborhood, and is non-negative, taking the logarithm is safe as everything is analytic. The interested reader can consult e.g. \cite[Section II.12.8]{ShiryaevProbability} which dissects this in detail. 
\begin{lemma} \label{lem:analytic_H}
    Assume that the cumulant-generating function $K_Y(t) = \log \E e^{tY}$ is finite on a neighborhood $(-\epsilon_0,\epsilon_0)$ of zero. Then,
    $K_Y(z) = \log \E e^{zY}$ as a function on the complex plane is holomorphic on the open ball $B_{\epsilon}(0)$ of radius $\epsilon$ centered at zero, for some $\epsilon > 0$. Moreover, for $z \in B_{\epsilon}(0)$, we have that
    \begin{align} \label{eq:expansion_lemma}
        K_Y(z) = \sum_{n = 1}^{+\infty} \frac{\kappa_n}{n!} z^n.
    \end{align}
\end{lemma}

\begin{lemma}[End of the proof of \cref{thm: main}] \label{lem: o_limit}
    We have that
    \begin{align}
        |\kl{\rho_t}{\pi} - \frac{\kappa_2}{2} e^{-2t}| = O(e^{-3t}), \qquad |\cR_q(\rho_t\| \pi) - \frac{q \kappa_2}{2}e^{-2t}| = O(e^{-3t}).
    \end{align}
\end{lemma}
\begin{proof}
    \cref{lem:analytic_H} implies that the series for $K_Y$ centered at zero has convergence radius $\epsilon$, for some $\epsilon > 0$. Since the derivative of a series has the same radius of convergence, we obtain that 
    \begin{align*}
        H(z) := z K_Y'(z) - K_Y(z) = \sum_{n\geq 2} \frac{\kappa_n}{n(n-2)!} z^n.
    \end{align*}
    has convergence radius $\epsilon$ as well.
    Hence, by the Cauchy-Hadamard theorem, $\frac{1}{\epsilon} \geq \limsup_{n \to \infty} (|c_n|^{1/n})$, where $c_n := \frac{\kappa_n}{n(n-2)!}$.
    
    This implies that for all $0 < \epsilon' < \epsilon$, there exists a constant $C_{\epsilon'} > 0$ such that for all $n \geq 0$, $|c_n| \leq C_{\epsilon'}/(\epsilon')^n$. Consequently, for all $z \in \mathbb{C}$ with $|z| < 1/\epsilon'$,
    \begin{align}
        |H(z) - \frac{\kappa_2}{2} z^2| = \big|\sum_{n = 3}^{+\infty} \frac{\kappa_n}{n(n-2)!} z^n \big| \leq C_{\epsilon'} \sum_{n = 3}^{+\infty} \big(\frac{|z|}{\epsilon'} \big)^n = C_{\epsilon'} \frac{\big(\frac{|z|}{\epsilon'} \big)^3}{1-\frac{|z|}{\epsilon'}}
    \end{align}
    Using \cref{eq:KL_K_Y}, we get that for any constant $\gamma > 0$, if $t \geq - \log \epsilon' + \gamma$ (or equivalently, $e^{-t} \leq \epsilon' e^{-\gamma}$),
    \begin{align}
        |\kl{\rho_t}{\pi} - \frac{\kappa_2}{2} e^{-2t}| \leq C_{\epsilon'} \frac{\big(\frac{e^{-t}}{\epsilon'} \big)^3}{1-\frac{e^{-t}}{\epsilon'}} = C_{\epsilon'} \frac{e^{-3t}}{(\epsilon')^3(1-e^{-\gamma})} = O(e^{-3t}),
    \end{align}
    which concludes the proof for the KL divergence. For the Rényi divergence, the proof is analogous (note that in that case the series $\frac{1}{q-1}K_Y(q z) - \frac{q}{q-1}K_Y(z)$ has convergence radius $\epsilon/q$).
\end{proof}

\section{Details on the numerical simulations} \label{sec:details_sim}
To run the simulations in \Cref{sec:simulations}, we discretized the interval $[-\pi,\pi)$ in $n=2000$ equispaced points. Let $h = 2\pi/n$. For each algorithm and initialization, we construct sequences ${(x_k)}_{k \geq 0}$, where $x_k \in \R^{n}$ represents the normalized log-density at each point. We let $v_* \in \R^{n}$ be the (non-normalized) energy of the target distribution, obtained by evaluating $V_*$ at the discretization points. Similarly, $\nabla v_*, \Delta v_* \in \R^{n}$ are the evaluations of $\nabla V_*$ and $\Delta V_*$ at the $n$ points (note that $\nabla V_*$ is a scalar because the distributions are one-dimensional).

We used the following discretizations for the Fisher-Rao, Wasserstein and Wasserstein-Fisher-Rao gradient flows:
\begin{enumerate}[label=(\roman*)]
    \item Fisher-Rao GF: We use mirror descent in log-space. The update reads:
    \begin{align*}
        \tilde{x}_{k+1} &\gets x_{k} + \epsilon (-v_*-x_k), \\
        x_{k+1} &\gets \tilde{x}_{k+1} - \log\bigg(\sum_{i=1}^{n} e^{-\tilde{x}_{k+1}^i} \bigg).
    \end{align*}
    \item Wasserstein GF: We approximate numerically the gradient and the laplacian of the log-density:
    \begin{align} 
    \begin{split} \label{eq:wasserstein_updates}
        \forall i \in [n], \qquad (\nabla x_k)^i &\gets (x_k^{i+1}-x_k^{i-1})/(2h), \\
        \forall i \in [n], \qquad (\Delta x_k)^i &\gets (x_k^{i+1}+x_k^{i-1}-2x_k^i)/h^2, \\
        x_{k+1} &\gets x_k + \epsilon (\Delta v_* + \Delta x_k + (\nabla v_* + \nabla x_k) \nabla x_k).
    \end{split}
    \end{align}
    We use periodic boundary conditions, so that the first discretization point is adjacent to the last one for the purposes of computing derivatives.
    \item Wasserstein-Fisher-Rao GF: We combine the two previous updates. Letting $\nabla x_k$ and $\Delta x_k$ be as in \cref{eq:wasserstein_updates}, we have
    \begin{align*}
        \tilde{x}_{k+1} &\gets x_{k} + \epsilon (-v_*-x_k + \Delta v_* + \Delta x_k + (\nabla v_* + \nabla x_k) \nabla x_k), \\
        x_{k+1} &\gets \tilde{x}_{k+1} - \log\bigg(\sum_{i=1}^{n} e^{-\tilde{x}_{k+1}^i} \bigg).
    \end{align*}
\end{enumerate}
We used stepsizes $\epsilon = \num{2.5e-6}$ and $\epsilon = \num{1e-6}$ for the experiments on target distributions (1) and (2), respectively. The slopes in \cref{table:slopes} are obtain by taking $0 < t_1 < t_2$ and computing
\begin{align*}
\frac{\log(\kl{\rho_{t_2}}{\pi})-\log(\kl{\rho_{t_1}}{\pi})}{t_2-t_1}.
\end{align*}
We use different values for $t_1$ and $t_2$ for each target distribution; $t_1$ and $t_2$ must be large enough to capture the asymptotic slope of the curve, but not too large to avoid numerical errors. For all the curves corresponding to target $\pi_1$, we take $t_1 = 7.0$ and $t_2 = 7.5$. For target $\pi_2$, we take: for FR, $t_1 = 6.875$ and $t_2 = 7.0$; for WFR, $t_1 = 1.875$ and $t_2 = 2.0$; for W, $t_1 = 2.75$ and $t_2 = 2.875$. 

\end{document}